\documentclass{amsart} 
\usepackage{amssymb}

\newcommand\cle\preceq
\newcommand\cwedge\curlywedge
\newcommand\cvee\curlyvee

\newcommand\colo{\colon\,}
\newcommand\ran{\mathop{\mathrm{ran}}}
\newcommand{\stminus}
{\mathrel{^*\hspace{-.64em}-}}
\newcommand\thitem[1]{\item[\hspace*{-3ex}{\upshape (#1)}]}

\theoremstyle{plain}
\newtheorem{theorem}{Theorem}[section]
\newtheorem{corollary}[theorem]{Corollary} 
\newtheorem{lemma}[theorem]{Lemma}
\newtheorem{proposition}[theorem]{Proposition}
\theoremstyle{definition}
\newtheorem{remark}{Remark}
\newtheorem{example}{Example}

\author[J.\, C\={\i}rulis]{J\=anis C\={i}rulis}
\address{Faculty of Computing \\ University of Latvia \\ Latvia}
\email{jc@lanet.lv}
\title[Lattice operations on Rickart *-rings]{Lattice operations on Rickart *-rings \\ under the star order}

\begin{document}
\begin{abstract}
Various authors have investigated properties of the star order (introduced by M.P.\ Drazin in 1978) on algebras of matrices and of bounded linear operators on a Hilbert space. Rickart involution rings (*-rings) are a certain algebraic analogue of von Neumann algebras, which cover these particular algebras. In 1983, M.F.~Janowitz proved, in particular, that, in a star-ordered Rickart *-ring, every pair of elements bounded from above has a meet and also a join. However, the latter conclusion seems to be based on some wrong assumption. We show that the conclusion is nevertheless correct, and provide equational descriptions of joins and meets for this case. We also present various general properties of the star order in Rickart *-rings, give several necessary and sufficient conditions (again, equational) for a pair of elements to have a least upper bound of a special kind, and discuss the question when a star-ordered Rickart*-ring is a lower semilattice. 
\end{abstract}

\maketitle

\section{Introduction}

In \cite{D}, M.P.~Drazin proved that a certain relation on the so called proper involution rings (in fact, even involution semigroups) is an order, now known as the star order (or *-order, for short).  Properties of star order have most intensively been studied in the space of complex matrices and the ring of all bounded linear operators on an infinitely-dimensional Hilbert space. In particular, a number of alternative characteristics of the star order for these structures are found in the literature. Also, existence of joins and meets in various particular cases have attracted interest. Thus, in the early papers \cite{H,HD}, it was proved that *-regular involution rings have the upper bound property, resp., complex $m \times n$ matrices form a lower semilattice under the star ordering. (A poset is said to have the \emph{upper bound property} if every pair of its elements bounded from above has the least upper bound.) Later, S.~Gudder introduced in  \cite{G} a certain order for self-adjoint bounded operators on a complex Hilbert space; it is actually a particular star order and has been intensively investigated. Gudder noticed that every initial segment in the poset of such operators is a lattice and, hence, has the upper bound property (but see below a comment on this latter conclusion). S.~Pulmannov\'a and E.~Vincekov\'a \cite{PV} improved his result by showing that the poset of operators is \emph{bounded complete}, i.e. that every subset bounded from above has the supremum, and every nonempty subset has the infimum; see \cite{C} for another approach to this result. Existence conditions for joins and meets  of self-adjoint operators, and representations of these operations have been discussed in the physical literature (see references in \cite{C}. For related questions in certain abstract structures, see \cite{{LX}}.).  The star-ordered set of all bounded linear operators on a Hilbert space also has the upper bound property: this was shown by X.-M. Xu e.a.\ in \cite{XDKL}; a matrix representation for  joins of operators was also given there. Recently J.Antezana e.a.\ have extended the result of \cite{HD}: it was proved in \cite{ACMS} that any two bounded linear operators on a Hilbert space have the meet. Joins and meets of matrices under the *-order are discussed in \cite{MBM}.

These, and several other well-known structures (cf.\ \cite{B,Bo}) are examples of star-ordered Rickart *-rings and even Baer *-rings, or their Hermitian parts (see Section \ref{rick} for definitions). In \cite{J2}, M.F.~Janowitz proved that in  a Baer *-ring $R$ every nonempty subset has an infimum. It is an immediate consequence that then every subset bounded from above has a supremum. In particular, $R$ is a lower semilattice in which every initial segment is a lattice; such ordered structures are known as \emph{nearlattices}. These relatively old results in fact cover those  mentioned above. 
 
A sub-*-ring of a Baer *-ring is not necessary a Baer *-ring. Janowitz proved in \cite{J2} also that every initial segment of an arbitrary star-ordered  Rickart *-ring $R$ is an orthomodular lattice, and drew an immediate consequence that such a ring has the upper bound property. However, this conclusion seems to be based on a wrong assumption that the upper bound of two elements in an initial segment of a poset is their least upper bound in the whole poset (curiously, the same inaccuracy has been admitted in \cite{G}; see \cite{C})); we show below that it fortunately is correct. Another result of \cite{J2} is that $R$ is a lower semilattice if it has no nonzero nilpotent elements.  

The present paper is organized as follows. Section \ref{rick} contains necessary definitions, two examples, and some elementary properties of Rickart *-rings. Preliminary results on star ordering are presented in Section \ref{star}. In particular, we give there several alternative characterizations of the star order in a Rickart *-ring. Main results are concentrated in Section \ref{main}. We first give a simple proof of the result of \cite{J2} that bounded pairs of elements have both the join and the meet; the proof provides also an equational description of these. Then we discuss existence of joins and meets of some special bounded pairs of elements. At last, in Section \ref{meets} we turn to meets of arbitrary pairs of elements of a Rickart *-ring.

\section{Rickart *-rings} \label{rick}

An (associative)  involution ring ({*-ring}, for short)
$R$ is a \emph{Rickart *-ring} if the right annihilator of every element is generated by a projection, i.e., an idempotent and self-adjoint element. Thus, to each $x \in R$ there is a projection $x'$ such that,
\[
\mbox{for all $y \in R$, $xy = 0$ if and only if $y = x'y$};
\]
this $x'$ is necessarily unique. In particular, $R$ has the unit $1 := 0'$. The element $x''$ is sometimes called the \emph{right projection of $x$}. If the right annihilator of  any subset of $R$ is generated by a projection, the ring is called a \emph{Baer *-ring}.
(We follow the notation of \cite{F1,F2,J1,J2}; many authors write $x'$ for the commutant of $x$, and RP(x), for the right projection of $x$. See, for example, \cite{B}, which is the standard reference book on Rickart and Baer rings and *-rings.) 

Through the paper, let $R$ be a fixed Rickart *-ring. 
Its subset $P$ of \emph{projections}, i.e., idempotent and symmetric elements, is partially ordered by the relation $\le$ defined by $e \le f$ iff $ef = e$ or, equivalently, $fe = e$; $0$ is its least, and 1, the greatest element. Projections form an orthomodular lattice, where $e'$ is the orthocomplement of $e$ \cite{B,F1,F2,J1}. It is easily seen that $e'= 1 - e$. 
We let the usual symbols $\vee$ and $\wedge$ stand for the lattice operations in $P$. For example, $ef \in P$ if and only if $e$ and $f$ commute, and then $e \wedge f = ef$; likewise, $e + f \in P$ if and only if $ef = 0$, and then $e \vee f = e + f$. Notice that every initial segment $[0,g]$ of $P$ is a sublattice, which also is orthomodular with $g \wedge e'$ ($= g - e$) the orthocomplement of $e$. $R$ is a Baer *-ring if and only if its lattice of projections is complete \cite{B}.

\begin{lemma}	\label{C-latt}
For every $x \in R$, the subset $C(x) := \{e \in P\colo ex = xe\}$ is a subortholattice of $P$. It is closed also under all infinite joins and meets existing in $P$. 
\end{lemma}
\begin{proof}
Evidently, $0,1 \in C(x)$. By Theorem 1(ix) of \cite{F2}, the subset $C(x)$  is closed under orthocomplementation ${}'$. It follows from Theorem 2(iv) of \cite{F2} that it is closed also under arbitrary joins (hence, also meets). 
\end{proof}

Elements $x$ and $y$ of $R$ are said to be \emph{*-orthogonal} (in symbols, $x \perp y$), if they satisfy any of the following equivalent equations:
\[
xy^* = 0 = x^*y, \quad yx^* = 0 = x^*y, \quad yx^* = 0 = y^*x, \quad xy^* = 0 = y^*x .
\]
In particular, $x \perp y$ if and only if $x^* \perp y^*$. 
Projections $e$ and $f$ are orthogonal if and only if $ef = 0$ or, equivalently, $fe=0$. Evidently, this is the case if and only if $f \le e'$ (resp., $e \le f'$). We shall use these characterizations of *-orthogonality without explicit references.

\begin{example}{\cite{F2,B}} \label{B(H)}
Let $H$ be a Hilbert space, and let $\mathcal B(H)$ be the ring of all bounded linear operators on $H$. The transfer $^*$ from an operator to its adjoint is an involution, and $\mathcal B(H)$ is actually a Baer *-ring. The projections in this ring are just the orthogonal projections onto closed linear subspaces of $H$; recall that such projections form a complete orthomodular lattice isomorphic to the lattice of all closed subspaces of $H$. Operators $A$ and $B$ are *-orthogonal if and only if closures of their ranges are orthogonal subspaces of $H$. 

To interpret the operations $'$ and $''$, it is convenient to regard operators as operating on the right. Then, for $A \in \mathcal B(H)$, $A'$ is the projection onto the orthogonal complement of the range $\ran A$ of $A$, $(A^*)'$ is the projection onto the nullspace $\ker A$ of $A$, $A''$ is the projection onto the closure of $\ran A$, and $(A^*)''$ is the projection onto the orthogonal complement of $\ker A$.
\end{example}

\begin{example}{\cite{B,H}} 
A ring is said to be \emph{regular} if every its principal right ideal is generated by an idempotent. This is the case if and only if, for every its element $x$, there is an element $u$ (an \emph{inner generalized inverse} of $x$) such that $xux = x$. Suppose that a ring $A$ is  *-regular, i.e., a regular *-ring with proper involution \cite{B}. Then  a generalized inverse of $x$ coincides with its Moore-Penrose inverse $x^\dag$ and is uniquely defined. Further, $xy = 0$ if and only if $(1 - x^\dag x)y = y$. Moreover, $1 - x^\dag x$ is a projection.
Therefore $A$ is a Rickart *-ring with $x'= 1 - x^\dag x$ for every $x$; likewise, $(x^*)'= 1 - xx^\dag$. Further, $x'' = x^\dag x$ and $(x^*)'' = xx^\dag$.
\par
Note that complex $n \times n$ matrices form a *-regular ring with the conjugate transposition in the role of involution.
\end{example}

We now list some elementary properties of the operation $'$. Let $K_x := C(x^*x)$.
\begin{proposition}	\label{list2}
In $R$,
\begin{enumerate}
\thitem{a}
$aa'= 0 = a'a^*$,
\thitem{b}
$aa'' = a = (a^*)''a$, 
\thitem{c}
$ab = 0$ iff $a''b = 0$,
\thitem{d}
 $a \perp b$ iff ${a^*}''b = 0 = ba''$,
\thitem{e}
$x'' \in K_x$,
\thitem{f}
$(a''b)'' = (ab)'' \le b''$, 
\thitem{g}
if $e \le a''$, then $(ae)'' = e$,
\thitem{h}
$\{e \in P\colo ae = 0\}$ is a sublattice of $P$. 
\end{enumerate}
\end{proposition}
\begin{proof}
(a) Use the definition of $'$, and apply $^*$.
\par
(b) By (a), $a''a^* = a^*$.
\par
(c) As $a'''= a'$.
\par 
(d) By (c), $ab^* = 0$ iff $a''b^* = 0$ iff $ba'' = 0$. Likewise,  $a^*b = 0$ iff ${a^*}''b = 0$. 
\par
(e) By (b).
\par
(f) See Theorem 1(xii,xiv) in \cite{F2}. 
\par
(g) If $a''e = e$, then, by (f), $(ae)'' = (a''e)'' = e'' = e$.
\par
(h) For $\{e \in P\colo ae = 0\} = [0,a']$. 
\end{proof}

By the \emph{Hermitian part} of $R$ we shall mean its subset $S := \{x \in R\colo x^* = x\}$ of self-adjoint elements. Evidently, $S$ contains $0$, $1$, all projections, and is closed under operations $+$ and $'$. Moreover, if $x,y \in S$, then $xy \in S$ if and only if $xy = yx$. Observe that $x \perp y$ in $S$ if and only if any of the products $xy$, $yx$, $xy''$, $x''y$, $y''x$, $yx''$ and $x''y''$ equals to $0$ (see P\ref{list2}) and that $C(x) \subseteq K_x$ whenever $x \in S$.

\section{Star order}	\label{star}

We write $x \cle y$ to mean that $x \perp y - x$. This is the case if and only if elements $x$ and $y$ of $R$ satisfy any of the following four equivalent pairs of conditions:
\begin{gather*}
x^*x = x^*y \text{ and } xx^* = yx^*, \quad 
x^*x = y^*x \text{ and } xx^* = xy^*, \\
x^*x = x^*y \text{ and } xx^* = xy^*, \quad
x^*x = y^*x \text{ and } xx^* = xy^*. 
\end{gather*}
In $S$, they reduce to the simple equation $x^2 = xy$ (or to an equivalent one, $x^2 = yx$). Evidently, $x \cle y$ if and only $x^* \cle y^*$. Below, we shall use these characterizations of the relation $\cle$  without explicit references. 

It was observed in \cite{D} that if the involution in $R$ is proper, i.e., satisfies the *-cancellation law
\[ 
x^*x = 0 \text{ only if } x = 0,
\] 
then the relation $\cle$  is an order on $R$; it is now known as the \emph{star order} (or \textit{*-order}). (Actually, the converse also holds.) In the sequel, $R$ is assumed to have a proper involution.

We now list several elementary but useful properties of the star order. Let $\cvee$ and $\cwedge$ stand for (generally, partial) lattice operations in $R$. 

 \begin{lemma}   \label{list1}
In $R$,
\begin{enumerate}
\thitem{a} $0$ is the least element in $R$, 
\thitem{b} the star order coincides on $P$ with the usual order of projections,
\thitem{c} $a \in P$ if and only if $a \cle 1$,
\thitem{d} every right (left) invertible element is maximal,
\thitem{e} if $e \le f$ and $e \in K_x$, then $xe \cle xf$,
\thitem{f} 
$K_x = \{e\colo xe \cle x\}$,
\thitem{g} if $a \perp b$, then $a \cwedge b = 0$ and $a \cvee b = a + b$.
\end{enumerate}
\end{lemma}
\begin{proof}
(a) evident.
\par
 (b) In fact, $e^*e = e^*f$ and $ee^* = fe^*$ if and only if $e = ef = fe$.
 \par
(c)
By definition of $\cle$, $a \cle 1$ iff $a^*a = a = aa^*$ iff $aa^* = a^* = a^*a$. This is the case if and only if $a = a^*$ and $a^2 = a$.
\par
(d)
If $ay = 1$ and $a \cle z$, then $y^*a^* = 1$, $a^*a = a^*z$ and $a = y^*a^*a = y^*a^*z = z$. Likewise, if $ya = 1$.
\par
(e) If $e \le f$, then $xf(xe)^* = xfex^* = xeex^* = xe(xe)^*$. If, moreover, $e \in K_x$, then $(xe)^*xf = ex^*xf = x^*xef = ex^*xe = (xe)^*xe$.
\par
(f) See the equivalence (11) in \cite{J2}.
\par
(g)
 This is Lemma 1 of \cite{J2}.
\end{proof}

Apart from items (1b) and (2b), which go back (for matrices) to \cite{He}, the next lemma is an abstract algebraic version of \cite[Proposition 2.3]{ACMS}, which deals with bounded Hilbert space operators.

\begin{lemma}		\label{cle:l}
\begin{enumerate}
\thitem{1} For $a,b \in R$, the following are equivalent:
\begin{enumerate}
\thitem{a}
$aa^* = ba^*$,
\thitem{b} $b = a + c$ for some $c$ with $ca^* = 0$,
\thitem{c}
$a = ba''$,
\thitem{d}
$a = be$ for some $e \in P$.
\end{enumerate}
\thitem{2} For $a,b \in R$, the following are equivalent:
\begin{enumerate}
\thitem{a}
$a^*a = a^*b$,
\thitem{b}
$b = a + c$ for some $c$ with $a^*c = 0$,
\thitem{c}
$a = (a^*)''b$,
\thitem{d}
$a = fb$ for some $f \in P$.
\end{enumerate}
\end{enumerate}
\end{lemma}
\begin{proof}
We shall demonstrate only (1).
\par
(a)$\to$(b):
put $c := b - a$; then $a + c = b$ and $ca^* = ba^* - aa^* = 0$.

(b)$\to$(c):
first observe that $ca^* = 0$ iff $ac^* = 0$ iff $a''c^* = 0$ iff $ca'' = 0$. 
Now, $ba'' = aa'' + ca'' = a + 0 = a$.

(c)$\to$(d):
put $e := a''$.

(d)$\to$(a):
$a^* = eb^*$, and $aa^* = beeb^* = ba^*$.
\end{proof}

\begin{theorem}  \label{cle!}
The following assertions are equivalent for all $a,b \in R$:
\begin{enumerate}
\thitem{a} $a \cle b$,
\thitem{b} $b = a + c$ for some $c \perp a$,
\thitem{c} $ba'' = a = (a^*)''b$,
\thitem{d} $fb = a = be$ for some $e,f \in P$, 
\thitem{e} $a = ba''$ and $a''$ commutes with $b^*b$,
\thitem{f} $a = ba''$ and $a^*b \in S$.
\end{enumerate}
\end{theorem}
\begin{proof}
Equivalence of (a), (b), (c) and (d) follows immediately from the preceding lemma.
Assume (a); then $a''b^*b = (ba'')^*b = a^*b = a^*a = b^*a = b^*ba''$ by (c), and (e) follows. Assume (e); then $ba'' \cle b$ by L\ref{list1}(f), and (a) follows. At last, (f) is equivalent to (e):  the equality $a = ba''$ implies that $a'' \in K_b$ if and only if $a^*b = b^*a$.
\end{proof}
 
Equivalence of (a), (e) and (f) was observed already in \cite{J2} (see there the paragraph subsequent to the proof of Lemma 4).

\begin{corollary}	\label{''}
If $a \cle b$, then $b'\le a'$ and $a'' \le b''$.
\end{corollary}
\begin{proof}
If $a \cle b$, then $0 = (a^*)''bb'= ab' = a''b'$ (see T\ref{cle!}(c) and P\ref{list2}(a,c)). Hence, $a'' \perp b'$ and $a'' \le b''$. In the ortholattice $P$, the last inequality implies that $b'\le a'$.
\end{proof}

\begin{remark}
Theorem 5 in \cite{DM} presents a description of the star order in $\mathcal B(H)$, which, in the notation of our Example \ref{B(H)}(with operators operating on the right!), states that $A \cle B$ if and only if there are projection operators $P$ and $Q$ such that $\ran P = \overline{\ran A}$ (i.e. $P = A''$), $\ker A = \ker Q$ (i.e., ${A^*}' = {Q^*}' = Q'$), $AP = BP$ and $QA = QB$.
Since $AA'' = A = {A^*}''A$ (see P\ref{list2}(b)), this amounts to the equality ${A^*}''B = A = BA''$. Therefore, the equivalence of (a) and (c) in Theorem \ref{cle!} may be considered as a simple abstract analogue of the mentioned description and provides the latter with a  short algebraic proof.
\end{remark}

\begin{remark}
The left-star and right-star partial orders  \cite{BM,DW,MBM}, defined first for matrices, have been applied also to operators in $\mathcal B(H)$. Under the conventions of Example \ref{B(H)}, they are described respectively by  
\begin{gather*}
A\, {*{\cle}}\, B :\equiv AA^* = BA^* \text{ and }\ran A \subseteq \ran B, \\
A\, {{\cle}*}\, B :\equiv A^*A = A^*B \text{ and }\ran A^* \subseteq \ran B^*.
\end{gather*}
Observe that $\ran A \subseteq \ran B$ iff $(\ran B)^\perp \subseteq (\ran A)^\perp$ iff $B' \le A'$ iff $A'' \le B''$, where $X^\perp$ stands for the orthogonal complement of $X$, and likewise $\ran A^* \subseteq \ran B^*$ iff $(A^*)'' \le (B^*)''$. So, the relationships 
\[ 
x\, {*{\cle}}\, y :\equiv xx^* = yx^* \text{ and } x'' \le y'', \
\text{resp.}, \
x\, {{\cle}*}\, y :\equiv x^*x = x^*y \text{ and } (x^*)'' \le (y^*)''
\]
introduce the left-star and right-star orders in abstract Rickart *-rings. The descriptions can be simplified applying Lemma \ref{cle:l}. 
\end{remark}

The above theorem leads us to the following abstract algebraic version of \cite[Lemma 4.3]{G} dealing with self-adjoint Hilbert space operators.
\begin{corollary}	\label{cle!!}
The following assertions are equivalent for $a,b \in S$:
\begin{enumerate}
\thitem{a} $a \cle b$,
\thitem{b} $b = a + c$ for some $c \in S$ with $c \perp a$,
\thitem{c1} $a = ba''$,
\thitem{c2} $a = a''b$,
\thitem{d1} $a = be$ for some $e \in P$,
\thitem{d2} $a = fb$ for some $f \in P$.
\end{enumerate}
\end{corollary}
\begin{proof}
In virtue of T\ref{cle!}, it suffices to show that (d1) and (d2) are equivalent and imply (a). Assume (d1); then $eb = (be)^* = a^* = a$. Now, $a^2 = beeb = beb = ab$. Likewise, (d2) implies (d1) and (a).
\end{proof}

In particular, if $a,b \in S$, then $a''b = ba'' \in S$ whenever $a \cle b$.

\section{Joins and meets of bounded pairs of elements} \label{main}

As the unit 1 is invertible, it follows from L\ref{list1}(d) that $1 \cvee a$ is not defined unless $a \cle 1$. Then L\ref{list1}(c) leads us to the following conclusion, which shows that, except for trivial cases, $R$ never is an upper semilattice.
\begin{corollary}
$R$ is an upper semilattice if and only if $P = R$.
\end{corollary}

However, $R$ may have the upper bound property or be a lower semilattice. The subsequent result is 
a direct analogue of \cite[Theorem 6]{C} stated for self-adjoint Hilbert operators. (The demonstration of the theorem in \cite{C}  contained a non-algebraic argument.) 

\begin{theorem}  \label{ab<x}
If $a,b \cle x$, then
\begin{enumerate}
\thitem{a} $a \cwedge b$ exists and equals to $x(a'' \wedge b'')$,
\thitem{b} $a \cvee b$ exists and equals to $x(a'' \vee b'')$.
\end{enumerate}
If, moreover, $a,b,x \in S$, then $a \cwedge b, a \cvee b  \in S$.
\end{theorem}
\begin{proof}
 Suppose that $a,b \cle x$ and, consequently, $a = xa'', b = xb''$  and $a'', b'' \in K_x$ (T\ref{cle!}(e)).
 \par
(a) Then $a'' \wedge b'' \in K_x$ (L\ref{C-latt}) and, by L\ref{list1}(e), $c := x(a'' \wedge b'')$ is a lower bound of $a$ and $b$. Suppose that $d$ is one more lower bound; then $d'' \le a''$, $d'' \le b''$ (C\ref{''}) and $d'' \le a'' \wedge b''$. Moreover, $d \cle x$ and, consequently, $d = xd''$ and $d'' \in K_x$ (T\ref{cle!}(e)). Now $d \cle c$ by L\ref{list1}(e), and $c$ is the greatest lower bound of $a$ and $b$.
\par
(b) Likewise, $c := x(a'' \vee b'')$ is an upper bound of $a$ and $b$. Suppose that $y$ is one more upper bound, then $a = ya'', b = yb''$ and $a'',b'' \in K_y$ (T\ref{cle!}(e)); by L\ref{C-latt}, also $a'' \vee b'' \in K_y$. Hence, $d:= y(a'' \vee b'') \cle y$ by L \ref{list1}(f). Now $c = d$: as $(x-y)a'' = 0 = (x-y)b''$, P\ref{list2}(h) implies that $(x-y)(a'' \vee b'') =0$. Thus $c \cle y$, and therefore $c$ is the least upper bound of $a$ and $b$.
\par
At last, if $a,b, x \in S$, then, by L\ref{C-latt} together with C\ref{cle!!}(c1,c2), $a'' \wedge b'', a'' \vee b'' \in C(x)$. Therefore, $a \cwedge b, a \cvee b \in S$, indeed.
\end{proof}

Item (b) of the theorem confirms that $R$ has the upper bound property (see the relevant discussion in Introduction).  The theorem also implies that elements of $R$ having the join have also the meet. Therefore, every segment $[0,x]$ is a sublattice of $R$, i.e., it is a lattice under $\cle$ and joins and meets in $[0,x]$ are also joins, resp., meets in $R$. In particular, $P$ is a sublattice of $R$, i.e., the operations $\cvee$ and $\cwedge$ agree on $P$ with $\vee$, resp.,  $\wedge$ (cf.\ L\ref{list1}(b,c)). 

\begin{corollary} \label{cor<x}
If $a$ and $b$ have an upper bound, then 
\begin{enumerate}
\thitem{a}  
$(a \cvee b)'' = a'' \vee b''$ and $(a \cwedge b)'' = a'' \wedge b''$,
\thitem{b}  
\(
a(a'' \wedge b'') = a \cwedge b = b(a'' \wedge b'')
\). 
\end{enumerate}
\end{corollary}
\begin{proof}
(a) Suppose that $a,b \cle x$. Then $a'' \vee b'' \le x''$ (C\ref{''}) and $(a \cvee b)'' = (x(a'' \vee b''))'' = (x''(a'' \vee b'')'' = (a'' \vee b'')'' = a'' \vee b''$ (see P\ref{list2}(f)), 
and likewise $(a \cwedge b)'' = a'' \wedge b''$.
\par
(b) By (a) and T\ref{cle!}(c).
\end{proof}

 As to (a), cf.\ equations (14iii) in \cite{J2}. 
Further, the corollary implies that $S$ as well has the least upper bound property and that the join of two self-adjoint elements in $S$ is their join also in $R$.

The following generalisation of the theorem is proved similarly. 

\begin{theorem} \label{baer}
Suppose that $R$ is a Baer *-ring. If $x$ is an upper bound of a  nonempty subset $A \subseteq R$, then
\begin{enumerate}
\thitem{a} $x\bigwedge(a''\colo a \in A)$ is the greatest lower bound of $A$,
\thitem{b} $x\bigvee(a''\colo a \in A)$ is the least upper bound of $A$.
\end{enumerate}
If, moreover $A \subseteq S$ and $x \in S$, then the bounds both belong to $S$.
\end{theorem}

Elements $a$ and $b$ of $R$ are said to be \emph{coherent} if
\[
a^*b = b^*a \mbox{ and } ab'' = ba'' \mbox{ (or, equivalently, } a''b^* = b''a^*).
\]
Theorem 9 of \cite{J2} presents two complicated conditions that are necessary and sufficient for elements $a,b \in R$ to have both the join and meet correlated by the identity $a \cvee b = a + b - a \cwedge b$. We now show that the much simpler coherence condition does the job, and even more. In connection with items (c) and (d) in the subsequent theorem, see also the equation (16) and Corollary 10 in \cite{J2}.

\begin{theorem}
The following conditions on elements $a,b \in R$ are equivalent:
\begin{enumerate}
\thitem{a} 
$a$ and $b$ are coherent,
\thitem{b}
$a$ and $b$ have an upper bound, $a'' \in K_b$ and $b'' \in K_a$,
\thitem{c}
$a \cvee b$ exists, and $a + ba'= a \cvee b = b + ab'$,
\thitem{d}
$a \cwedge b$ and $a \cvee b$ exist, and
\par
$a b'' = a \cwedge b = ba''\!$, 
\quad 
$a \cvee b = ab' + a \cwedge b + ba'\; (= a + b - a \cwedge b)$.
\end{enumerate}
\end{theorem}

\begin{proof}
(a) $\to$ (b),(c).
Suppose that $a^*b = b^*a$ and $ab'' = ba''$. Then $a \cle a + ba'$, for $a \perp ba'$: by virtue of P\ref{list2}(a), $a(ba')^* = aa'b^* = 0$ and $(ba')^*a = a'b^*a = a'a^*b = 0$. Likewise, $b \cle b + ab'$. But $a = ab'+ ab''$ and $b = ba'+ ba''$; so, $a + ba'= b + ab'$. Hence, $a + ba'$ is an upper bound of $a$ and $b$.

Next, $a''b^*b = b''a^*b = b''b^*a = b^*a \mbox{ (P\ref{list2}(b))} = a^*b = a^*bb'' = b^*ab'' = b^*ba''$. Therefore, $a'' \in K_b$; likewise, $b'' \in K_a$. So, (b) holds. Observe that, by L\ref{C-latt}, also $a'\in K_b$ and $b' \in K_a$.

Further, $a + ba'$ is even the least upper bound of $a$ an $b$. Suppose that $a,b \cle x$ for some $x$. As $ba' \cle b$ by L\ref{list1}(f), then $ba' \cle x$.  We already know that $a \perp ba'$; by L\ref{list1}(g), then  $a + ba' = a \cvee ba'\cle x$. Likewise, $b +ab'$ is the least upper bound of $a$ and $b$.  Therefore, (c) also holds.

(b) $\to$ (a).
Suppose that $a,b \cle x$, so that $a = xa''$ and $b = xb''$ by T\ref{cle!}(c). If also $a'' \in K_b$ and $b'' \in K_a$, then $a'' \wedge b'' \in K_b$ (L\ref{C-latt}, P\ref{list2}(e)) and $b(a'' \wedge b'') \cle ba''$ by L\ref{list1}(e). On the other hand, $ba'' \cle b(a'' \wedge b'')$. Indeed, $(ba'')'' \le a''$ (P\ref{list2}(f)), $ba'' \cle b$ (L\ref{list1}(f)) and $(ba'')'' \le b''$(C\ref{''}); so $(ba'')'' \le a'' \wedge b''$. Also, $ba'' = b(ba'')''$ and $(ba'')'' \in K_b$ (T\ref{cle!}(e)).
It follows that $b(ba'')''  \cle b(a'' \wedge b'')$ (L\ref{list1}(e)) and, consequently, $ba'' \cle b(a'' \wedge b'')$. Thus, $ba'' = b(a'' \wedge b'') = a \cwedge b$ (see C\ref{cor<x}(b)), and likewise $ab'' = a \cwedge b$. Therefore, $ab'' = ba''$. 
Furthermore, $a^* = (xa'')^* = a''x^*$ by T\ref{cle!}(c), and likewise $b^* = b''x^*$; so
\(
a^*b = a''x^*xb'' = x^*xa''b'' = x^*ab'' = x^*ba'' = x^*xb''a'' = b''x^*xa'' = b^*a\).

(c) $\to$ (b). Suppose that $ a + ba' = a \cvee b = b + ab'$. But $a = ab' + ab''$ and $b = ba'+ ba''$; it follows that $ab'' = ba''$. Now let $x := a \cvee b$; then $a'',b'' \in K_x$ by T\ref{cle!}(e) and, further, $a'' \in K_b$: $a''b^*b = a''(xb'')^*xb'' = a'' b''x^*xb'' = x^*xa''b'' = x^*ab''= x^*ba'' = x^*xb''a'' = b''x^*xb''a'' = (xb'')^*xb''a'' = b^*ba''$. Likewise, $b'' \in K_a$.  

(b),(c) $\to$ (d).
We saw in the proof (b)$\to$(a) that (b) implies the identity $ab'' = a \cwedge b = ba''$. Theorem \ref{ab<x}(a) further implies  the first identity in (d). Then $a = ab' + (a \cwedge b)$ and $b = ba'+ (a \cwedge b)$; together with (c), this gives us the other identity.

(d) $\to$ (c). Obviously. 
\end{proof}

\section{When is $R$ a lower semilattice?} \label{meets}

Standard order-theoretic considerations show the following consequence of T\ref{baer}(b).

\begin{corollary}[\protect{\cite[Theorem 7]{J2} } ]
In a Baer *-ring, every nonempty subset has the greatest lower bound. 
\end{corollary}
  
In particular, any Baer *-ring is a lower semilattice. As $a \cwedge b$ is the least upper bound of all lower bounds of $A := \{a,b\}$, we conclude from T\ref{baer}(b) that
\[
a \cwedge b = a\bigvee(u''\colo u  \cle a, b) = b\bigvee(u''\colo u \cle a, b).
\]
We now adjust this result to an arbitrary Rickart *-ring $R$.

\begin{theorem}
Elements $a$ and $b$ of $R$ have the meet if and only if the set $\{u''\colo u \cle a,b\}$ has the greatest element $m$. If this is the case, then $m = (a \cwedge b)''$ and $a \cwedge b = am = bm$. 
\end{theorem}
\begin{proof}
If $a \cwedge b$ exists, then $a \cwedge b = a(a \cwedge b)''$ by T\ref{cle!}(c). Evidently, the element $(a \cwedge b)''$ belongs to $\{u''\colo u \cle a,b\}$. Moreover, it is the greatest element in this set: if $u \cle a,b$, then $u \cle a \cwedge b$ and $u'' \cle (a \cwedge b)''$ by C\ref{''}. 

Now suppose that the set $\{u''\colo u \cle a,b\}$ has the greatest element $u''_0$, where $u_0$ is a lower bound of $a$ and $b$. Then $bu''_0$ is the greatest lower bound of $a$ and $b$. Indeed, if $u \cle a,b$, then $u'' \in K_b$ (T\ref{cle!}(e)),  $u'' \cle u''_0$ and, further, $u = bu'' \cle bu''_0 = u_0$ (T\ref{cle!}(c) and L\ref{list1}(e)). Likewise, $au''_0$ also is such a lower bound.
\end{proof}

The involved set $\{u''\colo u \cle a,b\}$ admits a more detailed description. For any $a,b \in R$, let
\[
L_{a,b} :=  
\{e \in K_a \cap K_b\colo e \le a'' \wedge b'' \wedge (a - b)'\}.
\]
This definition is suggested by the proof of Proposition 3.2 in \cite{ACMS}. Cf.\ also the end of Section 4 in \cite{C} and the proof of \cite[Theorem 7]{J2}.

\begin{theorem}
$L_{a,b} = \{u''\colo u \cle a,b\}$.
\end{theorem}
\begin{proof}
If $u \cle a,b$, then, by T\ref{cle!}(e), $au'' = u = bu''$ and $u'' \in K_a \cap K_b$. Hence,  $(a - b)u'' = 0$, i.e., $u'' \le (a - b)'$. By C\ref{''}, $u'' \le a'' \wedge b''$; therefore, $u'' \in L_{a,b}$.

Conversely, if $e \in L_{a,b}$, then $ae \cle aa'' = a$  by Lemma \ref{list1}(e) and Proposition \ref{list2}(b), and likewise $be \cle b$. Also, $(a - b)e = 0$, whence $ae = be$. So, $ae$ is a lower bound of $a$ and $b$, and P\ref{list2}(g) implies that $e = (ae)'' \in\{u''\colo u \cle a,b\}$.
\end{proof}

We end the section with a characteristic of those Rickart *-rings which are lower semilattices with respect to the star order.

 \begin{lemma}	\label{meet2minus}
 If  $R$ is a lower semilattice under $\cle$, then the operation $\stminus$ defined by $ x \stminus y := x(x \cwedge y)' $ has the following properties:
\begin{enumerate}
\thitem{a}
$x - (x \stminus y) \cle y$,
\thitem{b}
if $x \cle y$, then $z \stminus y \cle z \stminus x$ for every $z$,
\thitem{c}
if $y \cle x$, then $x \stminus y = x - y$. \label{m/m}
\end{enumerate}
 \end{lemma}
 \begin{proof}
Evidently,  $x - (x \stminus y) = x(1 - (x \cwedge y)') = x(x \cwedge y)'' = x \cwedge y \cle y$ (T\ref{cle!}(c)). Further, if $x \cle y$, then $z \cwedge x \cle z \cwedge y \cle z$ and $(z \cwedge y)' \cle (z \cwedge x)'$ (C\ref{''}). Moreover, $(z \cwedge x)'', (z \cwedge y)'' \in K_z$ (T\ref{cle!}(e)), whence, $(z \cwedge x)', (z \cwedge y)' \in K_z$ by L\ref{C-latt}. It follows by L\ref{list1}(e) that $z \stminus y = z(z \cwedge y)'\cle z(z \cwedge x)'= z \stminus x$. 
At last, if $y \cle x$, then $y = xy'' \mbox{ (T\ref{cle!}(c))} = x(x \cwedge y)'' = x(1 - (x \cwedge y)')$, whence $x - y = x(x \cwedge y)' = x \stminus y$.
\end{proof}

We call any binary operation $\stminus$ on $R$ satisfying the conditions (a)--(c) a \emph{star minus}.

\begin{theorem}	\label{^-}
$R$ is a lower semilattice under star order if and only if it admits a *-minus operation. If this is the case, then $(x \cwedge y) + (x \stminus y) = x$.
\end{theorem}
\begin{proof}
Necessity of the condition follows from the lemma. 
Now assume that $\stminus$ is a *-minus operation on $R$. 
Then, for every $z$, $x \stminus z \cle x \stminus 0 = x - 0 = x$. In particular, $m := x - (x \stminus y) = x \stminus (x \stminus y) \cle x$,  and $m$ is a lower bound of $x$ and $y$. We are going to show that it is the greatest lower bound. 

Let $u \cle x,y$. Then $x \stminus y \cle x \stminus u \cle x$ and, further,  $x \stminus (x \stminus u) \cle x \stminus (x \stminus y)$, i.e., $x - (x \stminus u) \cle m$. At the same time, $x - (x \stminus u) = x - (x - u) = u$; so $u \cle m$ and $m = x \cwedge y $.
\end{proof}

Therefore, if a *-minus operation on $R$ exists, then it is uniquely defined. We do not discuss properties of this operation in more detail here, and only mention without proof that the system ($R, \stminus,0)$ turns out to be a so called commutative weak BCK-algebra (some relations between such algebras and star-ordered self-adjoint Hilbert space operators were noticed in \cite{C}).

\section*{Acknowledgment}
This work has been supported by Latvian Science Council, Grant No.\ 271/2012.


\begin{thebibliography}{12}

\bibitem{ACMS}
J.~Antezana, C.~Cano, I.~Mosconi, D.~Stojanoff,  
\textit{A note on the star order in Hilbert spaces}, 
Linear Multilinear Algebra 58 (2010), pp.\ 1037--1051.

\bibitem{B}
S.K.~Berberian,             
\textit{Baer *-rings},
Springer-Verlag, Berlin-Heidelberg-New York, 2011.

\bibitem{Bo}
M.~Bohata,   
\textit{Star order on operator and function algebras}, 
Publ.\ Math.\ 79 (2011), pp.\ 211--229.

\bibitem{BM}
J.K.~Baksalary, S.K.~Mitra,
\textit{Left-star and right-star partial ordering},
Linear Algebra Appl.\ 149 (1991), pp.\ 73--89.

\bibitem{C}
J.~C\={\i}rulis,
\textit{Further remarks on an order of quantum observables},
Math.\ Slovaca (in print; a preprint available as arXiv:1301.0640).

\bibitem{D}
M.P.~Drazin, 
\textit{Natural structures on semigroups with involution},
Bull.\ Math.\ Amer.\ Math.\ Soc.\ 84 (1978), pp.\ 139--141.

\bibitem{DW}                             
Ch.~Deng, Sh.~Wang,
\textit{On some characterizations of the partial ordering for bounded operators},
Math.\ Inequalities and Applications 12 (2012), pp.\ 619--630.

\bibitem{DM}
G.~Dolinar, J.~Marovt, 
\textit{Star partial order on B(H)}, 
Linear Algebra Appl.\ 434 (2011), pp.\ 319--326.

\bibitem{F1}
D.J.~Foulis,
\textit{Baer *-semigroups},
Proc.\ Amer.\ Math.\ Soc.\ 11 (1960), pp.\ 648--654.

\bibitem{F2}
D.J.~Foulis,
\textit{Relative inverses in Baer *-semigroups},
Mich.\ Math.\ J.\ 10 (1963), pp.\ 65--84.

\bibitem{G}                              
S.~Gudder,            
\textit{An order for quantum observables}, 
Math.\ Slovaca 56 (2006), pp.\ 573--589.

\bibitem{H}
R.E.~Hartwig,  
\textit{Pseudo lattice properties of the star-orthogonal partial ordering for star-regular rings}, 
Proc.\ Amer.\ Math.\ Soc.\ 77 (1979), pp.\ 299--303.

\bibitem{HD}
R.E.~Hartwig, M.P.~Drazin,             
\textit{Lattice properties of the *-order for complex matrices}, 
J.\ Math.\ Anal.\ Appl.\ 86 (1982), pp.\ 359--378.

\bibitem{He}
R.M.~Hestenes,
\textit{Relative Hermitian matrices},
Pacific J.\ Math.\ 11 (1961), pp.\ 224--245.
                              
\bibitem{J1}
M.F.~Janowitz, 
\textit{A note on Rickart rings and semi-Boolean algebras}, 
Algebra Univers.\ 6 (1976), pp.\ 9--12.

\bibitem{J2}
M.F.~Janowitz,  
\textit{On the *-order for Rickart *-rings}, 
 Algebra Univers.\ 16 (1983), pp.\ 360--369.
 
\bibitem{LX}
Y.~Li, X.M.~Xu,
\textit{The logic order on a generalized Hermitian algebra},
Repts.\ Math.\ Phys.\ 69 (2012) pp.\ 371--381.
 
\bibitem{MBM}
S.K.~Mitra, P.~Bhimasanharam, S.B.~Malik,
\textit{Matrix Partial Orders, Shorted Operators and Applications},
World Scientific, Singapore, 2010.

\bibitem{PV}                              
S.~Pulmannov\'a, E.~Vincekov\'a,  
\textit{Remarks on the order for quantum observables}, 
Math.\ Slovaca 57 (2007), pp.\ 589--600.

\bibitem{XDKL}                              
X.-M.~Xu, H.-K.~Du, Hong-Ke; X.~Fang, Y.~Li,            
\textit{The supremum of linear operators for the *-order},
Linear Algebra Appl.\ 433 (2010), pp.\ 2198-2207.
                              
\end{thebibliography}
\end{document}